\documentclass[12pt]{amsart}
\usepackage{graphicx}
\usepackage{amssymb}
\usepackage{epstopdf}
\usepackage{hyperref}

\def\R{\mathbf{R}}

\newtheorem{conj}{Conjecture}
\newtheorem{theorem}{Theorem}
\newtheorem{lemma}{Lemma}

\begin{document}

\title[Homogeneous Central Configurations]{Planar $N$-body central configurations with a homogeneous potential}
\author{Marshall Hampton}

\begin{abstract}
Central configurations give rise to self-similar solutions to the Newtonian $N$-body problem, and play important roles in understanding its complicated dynamics.  Even the simple question of whether or not there are finitely many planar central configurations for $N$ positive masses remains unsolved in most cases.  Considering central configurations as critical points of a function $f$, we explicity compute the eigenvalues of the Hessian of $f$ for all $N$ for the point vortex potential for the regular polygon with equal masses.  For homogeneous potentials including the Newtonian case we compute bounds on the eigenvalues for the regular polygon with equal masses, and give estimates on where bifurcations occur.  These eigenvalue computations imply results on the Morse indices of $f$ for the regular polygon.   Explicit formulae for the eigenvalues of the Hessian are also given for all central configurations of the equal mass 4-body problem with a homogeneous potential.  Classic results on collinear central configurations are also generalized to the homogeneous potential case. Numerical results, conjectures, and suggestions for future work in the context of a homogeneous potential are given.
\end{abstract}


\maketitle

\section{Introduction}

The classical dynamics of $N$ point particles with masses $m_i$ interacting via a central potential $U$ are given by:

$$m_i \ddot{q}_{i;j} = \frac{\partial U}{\partial q_{i;j}},  \ \ i \in \{0,\ldots N-1\}, \ \ j \in \{1,\ldots,d\}$$

where $q_i \in \R^{d}$ is the position of particle $i$ with components $q_{i;j}$, and 

$$ U = \sum_{i<k} m_i m_k/ r_{i,k}^{A-2}$$
is the potential with a real parameter $A > 2$, and $r_{i,k}$ is the distance between $q_i$ and $q_k$.   The case of Newtonian gravity is $A=3$ \cite{Newt}, and provides the primary motivation for studying this more general problem.  We can extend this potential to the case $A=2$ by using the logarithmic potential

$$U = \sum_{i<k} m_i m_k \log(r_{i,k})$$
which arises in a simplified model of fluid vortices \cite{HelmholtzV,kirchhoff_vorlesungen_1883,aref1992grobli}.  In the vortex model case the parameters $m_i$ represent the strength of a vortex rotation, and can be any real value.  However our main interest is the Newtonian case, for which we usually assume non-negative mass parameters.

In this article we focus our attention on configurations with the special property that each particle is accelerated towards the center of mass of the system at a rate uniformly proportional to its distance from the center of mass, i.e.

$$m_i \ddot{q}_{i;j} = \lambda (q_{i;j} - q_{C,j})$$
with $q_C = \frac{1}{M} \sum m_i q_i$ the center of mass, and $M = \sum m_i$ is the total mass.   Such configurations are called {\em central configurations} (as well as {\em permanent} or {\em stationary configurations} in some older literature).  In the planar case they also account for the relative equilibria, which are equilibria in a uniformly rotating reference frame.  Central configurations are important in the $N$-body problem for a number of reasons, including the study of multiple body collisions \cite{mcgehee1974triple, hulkower78, devaney1980triple,moeckel1981orbits} and the topology of the phase space for a fixed energy \cite{cabral1973integral, Easton75, Albouy93I, mccord1998integral}.  Understanding of the dynamics near central configurations was critical in the proof of chaotic behavior in the three-body problem \cite{moeckel1989chaotic}.

For more background on central configurations we highly recommend the recent summary by Moeckel \cite{LMS}, as well as earlier surveys \cite{saari80,moeckel_central_1990, elbialy1990collision, lacomba2003singularities, saari05}.

A longstanding open problem about central configurations is whether or not there are finitely many equivalence classes of central configurations for a particular choice of $N$ positive masses (usually restricted to the Newtonian case of $A=3$).  The most famous version of this problem further restricts the configurations to $\R^2$, and was highlighted by Stephen Smale as the sixth of his `Mathematical problems for the next century' \cite{smale98}. Smale himself considered the problem \cite{smale_topology_1970, smale_topology_1970-1}, and introduced a topological viewpoint that we will consider in the next section.  This problem was also formulated earlier by Wintner \cite{wintner41} and Chazy \cite{chazy18}, and highlighted more recently in \cite{albouy2012some}.  We follow the usual convention of considering two planar configurations equivalent if there is a direct isometry between them (i.e. an orientation-preserving rigid motion).

The difficulty of the finiteness problem was underscored by the discovery of a counter-example in the 5-body problem if a negative mass is allowed \cite{roberts_continuum_1999}.  This example has been extended to more general settings \cite{little13}.  The existence of positive dimensional sets of central configurations for some negative mass parameters makes many approaches using algebraic geometry challenging, since methods based on complex varieties are incapable of ruling these out.  Indeed, although there have been numerous successful applications of methods from classical and real algebraic geometry and tropical geometry to the finiteness problem \cite{moeckel_generic_2001, leandro_finiteness_2003, hampton_finiteness_2005, hampton_finiteness_2009, hampton_finiteness_2010, hampton_finiteness_2011, albouykaloshin}, we believe that solving the finiteness problem in general requires additional tools. 

The primary hypothesis of this manuscript is that studying the central configurations for a homogeneous potential (more general than the Newtonian) will more naturally develop mathematical tools that will advance the Newtonian case, analogously to the use of tools from complex analysis in solving real-analytic problems (e.g. contour integrals).  In particular we believe it would be valuable to develop a framework for central configurations in the limiting case of $A \rightarrow \infty$.

In what follows we consider central configurations as critical points of the function $$f = \frac{M I}{2} + \frac{U}{A-2},$$ where $I$ is the moment of inertia 
$$I = \sum_{i=1}^N m_i r_i^2$$
where  $r_i = |q_i|$ is the distance from the $i$th point to the origin.  If the center of mass is at the origin, then 
$$I = \frac{1}{M}\sum_{i<j} m_i m_j r_{i,j}^2.$$
Because the potential function $U$ is invariant under translation, all critical points of $f$ will have their center of mass at the origin.  In contrast to some other formulations of central configurations, our $f$ is homogeneous in the mass parameters but not in the distance variables.  We have in effect set a preferred scale from the beginning in order to have an unconstrained problem.  We find this approach simplest, but there are many other formulations of the problem \cite{laura1905sulle, albouy_paper_2003, LMS, ferrario2017central}.

The idea of studying a more general potential, even if we are mainly interested in the Newtonian case, is an old one \cite{synge1932apsides, lacomba1988origin, albouy_paper_2003, albouy_euler_2007,cors2014uniqueness}.  We would like to especially highlight the study of the behavior of central configurations for large values of $A$, which has not recieved much attention in the literature before.  

We briefly review the relevant topology for using Morse theory in the $N$-body problem. 

The configuration space we will use is $\mathcal{C}_N = (\mathbb{R}^{2N} \backslash \Delta)/S^1$, where $\Delta$ is the subset of collisions ($q_i = q_j$ for $i \neq j$) and the quotient is taken with respect to proper rotations around the origin treating $\mathbb{R}^{2N}$ as $(\mathbb{R}^{2})^{N}$.  The  function $f$ is well-defined on this quotient since both $I$ and $U$ are rotationally invariant.  

The simplest versions of Morse theory concern the behavior of a smooth function on a compact manifold, with the extra condition that the critical points of the function are nondegenerate (i.e. the Hessian is nondegenerate).  Although our function $f$ is not defined on a compact manifold (because of the removal of the set $\Delta$), this can be remedied without too much effort because the gradient of $f$ will always become outward pointing close to $\Delta$.  This was made precise by Shub \cite{shub1971appendix}.  Assuming $f$ is nondegenerate, its level sets change in topology at each critical value.  If the topology of the manifold is non-trivial, such changes in the level sets are inevitable.  The index of a critical point is the dimension of the largest subspace on which the Hessian of $f$ is negative definite.  Let the number of critical points with index $j$ be $n_j$, and encode this information in the Morse polynomial:

$$M(t) = \sum_{j=0}^{dim(\mathcal{C}_N)} n_j t^j$$
where $t$ is an auxiliary variable.  The Poincar{\'e} polynomial for a manifold is defined as $P(t) = \sum \beta_j t^j$, where $\beta_j$ is the $j$th Betti number.  Morse theory relates these polynomials by

$$M(t) = P(t) + (1+t)R(t),$$
where $R(t)$ is a polynomial with non-negative integer coefficients.  This not only puts a lower bound on the number of critical points, but also provides constraints on the possible $n_j$.

For more background on Morse theory we recommend \cite{milnor2016morse, bott1988morse}.

The Poincar{\'e} polynomial for the manifold $\mathcal{C}_N$ is $P_N(t) = \prod_{j=1}^{N-1} (1 + j t)$ \cite{arnold1969cohomology}.  The corresponding result in the spatial case was first computed in \cite{pacella}; in dimension $d$, $P_{N,d}(t) = \prod_{j=1}^{N-1} (1 + j t^{d-1})$ \cite{LMS}.

It is not completely clear how much the index of a critical point influences the more general dynamical behavior of orbits near the central configuration, although there are some results connecting these \cite{SunHu09, Baru14}.  Not many cases of exact calculations of linear stability are available.  The general problem was made precise by Andoyer \cite{andoyer1906main}.  Particular cases have been studied for three bodies \cite{gascheau1843,routh1875,roberts02,santoprete06}, four bodies \cite{rayl39,brumberg_permanent_1957, roberts13, sweatman2014orbits}, restricted cases (i.e. with one or more infinitesimal masses) \cite{pedersen1952stabilitatsuntersuchungen}, and polygonal and $N+1$ ring systems \cite{hall88,scheeres1991linear,moeckel_dom94, roberts_linear_2000, CabralSchmidt,barry10}; more could be done, especially numerically, in our setting of a variable exponent potential.  A particularly interesting example is the lower bound for instability of equal mass relative equilbria found by Roberts \cite{roberts_spectral_1999} in the Newtonian case.

A recent result of Montaldi \cite{Montaldi2015} uses only the existence of a minimum of $f$ to derive the existence of a large family of symmetric central configurations; in many settings this result could be strengthened using Morse theory (assuming the function $f$ is non-degenerate).  The existing upper and lower bounds for central configurations are far from sharp, despite some substantial effort \cite{llibre1990number, mccord1996planar, merkel2008morse, albouy2015open}.

\section{The regular polygon in the N-body problem} \label{regpolysec}

We can prove some properties of the Morse index of central configurations for the regular polygon in the $N$-body problem for varying $A$, and speculate on some others.  Quite a few results for regular polygon central configurations are known for the Newtonian ($A=3$) and vortex $(A=2)$ cases \cite{Longley1907,PerkoWalter,meyer1988bifurcations,celli2011polygonal}.

These configurations are well suited to polar coordinates, so we will express the position of the $i$th particle as $$q_i = (r_i \cos (\theta_i), r_i \sin(\theta_i))$$
For the regular polygon centered at the origin, all of the radii are equal ($r_i = r$ for some $r$) and $\theta_i = \frac{2 \pi i}{N}$ where we will index the particles starting at $i=0$.  In what follows we denote the evaluation of a function at the equal mass regular polygon by a superscript circle, e.g. $f^{\circ}$.  For the interparticles distances we define $p_{i,j} = r u_{i,j} = (r_{i,j})^{\circ}$, where
$$u_{i,j} = \sqrt{2 - 2 \cos \left [ \frac{2 \pi(j-i)}{N} \right ]} = 2 \sin(\frac{\pi |j-i|}{N})$$
are the distances between points $i$ and $j$ on the unit radius regular polygon.

As before we consider central configurations as critical points of the function $f = \frac{M I}{2} + \frac{U}{A-2}$.  To calculate derivatives of $f$ we will need the partial derivatives of the interparticle distances with respect to $r_i$ and $\theta_i$:

$$\frac{\partial r_{i,j}}{\partial r_i} = \frac{r_i - r_j \cos(\theta_j - \theta_i)}{r_{i,j}}$$

$$\frac{\partial r_{i,j}}{\partial \theta_i} = \frac{- r_i r_j \sin(\theta_j - \theta_i)}{r_{i,j}}.$$

Evaluated on the regular polygon, we have

$$\left ( \frac{\partial r_{i,j}}{\partial r_i} \right )^{\circ}  = \frac{1 -  \cos(\theta_j - \theta_i)}{u_{i,j}} = u_{i,j}/2$$

$$\left ( \frac{\partial r_{i,j}}{\partial \theta_i}\right )^{\circ}  = -r \frac{\sin(\theta_j - \theta_i)}{u_{i,j}} = -r\cos\left(\frac{\pi |j-i|}{N}\right).$$

The gradient of $f$ with respect to $r_i$ and $\theta_i$ has components

$$\frac{\partial f}{\partial r_i} = m_i M r_i - m_i \sum_{j \neq i} m_j r_{i,j}^{-A} (r_i - r_j \cos(\theta_j - \theta_i))$$

and

$$\frac{\partial f}{\partial \theta_i} = m_i \sum_{j \neq i} m_j r_{i,j}^{-A} r_i r_j \sin(\theta_j - \theta_i).$$

Evaluated at the equal mass regular polygon these are 

$$\left (\frac{\partial f}{\partial r_i}\right)^{\circ} = N r - r \sum_{j \neq i} p_{i,j}^{-A} (1 -  \cos(\theta_j - \theta_i))$$
$$\left (\frac{\partial f}{\partial \theta_i}\right )^{\circ} = r^2 \sum_{j \neq i} p_{i,j}^{-A} \sin(\theta_j - \theta_i) = 0.$$
where the second quantity is zero because $\sin(t)$ is odd and $p_{i,j}$ is even.

To be a critical point of $f$, $(\frac{\partial f}{\partial r_0})^{\circ} = 0$, which can be solved for the radius:

$$r = \left ( \frac{\sum_{j=1}^{N-1} u_{0,j}^{2-A}}{2 N} \right )^{1/A}$$

As $A$ increases, this radius increases to the limit value $r_{\infty}(N) = (2 \sin(\pi/N))^{-1}$, for which $p_{i,i+1} = 1$.  For $A=2$, the radius is equal to $\sqrt{\frac{N-1}{2 N}}$.

Now to compute the Morse index of $f$ for regular polygons we need the components of the Hessian of $f$.  In expressions involving indices $i$ and $j$, it is assumed that $i \neq j$.  In the final form shown for each expression we use the radius $r$ and unit polygon distances $u_{i,j}$ as much as possible.

$$\frac{\partial^2 f} {\partial r_i \partial r_j} = m_i m_j r_{i,j}^{-A-2} \left ( A(r_i - r_j \cos(\theta_j - \theta_i)) (r_j - r_i \cos(\theta_j - \theta_i)) + r_{i,j}^2 \cos(\theta_j-\theta_i)\right )$$

$$R_{i,j} := \left (\frac{\partial^2 f} {\partial r_i \partial r_j} \right )^{\circ} = p_{i,j}^{-A} \left [ \frac{A}{2}(1 - \cos\left ( \frac{2 \pi (j-i)}{N} \right )) + \cos\left ( \frac{2 \pi (j-i)}{N} \right ) \right ]$$
which simplifies to
$$R_{i,j}  = r^{-A} u_{i,j}^{-A} \left ( u_{i,j}^2 \frac{A-2}{4} + 1 \right )$$

$$\frac{\partial^2 f} {\partial r_i^2} = m_i (M - \sum_{j \neq i} m_j r_{i,j}^{-A-2} \left [ r_{i,j}^2 - A(r_i - r_j \cos(\theta_j - \theta_i))^2\right ] )$$

$$R_{i,i} := \left (\frac{\partial^2 f} {\partial r_i^2} \right )^{\circ} = N - \sum_{j \neq i} p_{i,j}^{-A}\left (1 - \frac{A p_{i,j}^2}{4 r^2}\right )$$
or
$$R_{i,i} = N \left ( 1+\frac{A}{2} - 2\frac{\sum_{j \neq i} u_{i,j}^{-A}}{\sum_{j \neq i} u_{i,j}^{-A+2}} \right ) = r^{-A} \sum_{j\neq i} u_{i,j}^{-A} \left ( u_{i,j}^2 (\frac{A}{4}+\frac{1}{2}) - 1 \right )$$
(for these identities we use the explicit formula for the radius $r$, e.g. to rewrite $N = r^{-A} \sum_{j=1}^{N-1} u_{0,j}^{2-A}/2$).

$$\frac{\partial^2 f} {\partial r_i \partial \theta_j} = -m_i m_j r_{i,j}^{-A-2} r_j \sin(\theta_j - \theta_i)\left [r_{i,j}^2 - A r_i (r_i - r_j \cos(\theta_j - \theta_i)) \right ] $$

$$W_{i,j} := \left (\frac{\partial^2 f} {\partial r_i \partial \theta_j} \right )^{\circ} = p_{i,j}^{-A}r \sin\left ( \frac{2 \pi (j-i)}{N} \right ) (\frac{A}{2} - 1) $$
or equivalently
$$W_{i,j} = r^{-A+1}\left (\frac{A}{2}-1 \right) u_{i,j}^{-A+1} \sqrt{1 - \frac{u_{i,j}^2}{4}}$$

$$\frac{\partial^2 f} {\partial r_i \partial \theta_i} = m_i \sum_{j \neq i} m_j r_{i,j}^{-A-2}r_j \sin(\theta_j - \theta_i) \left [ -A r_i(r_i - r_j \cos(\theta_j - \theta_i)) + r_{i,j}^2 \right ]$$

$$W_{i,i} := \left (\frac{\partial^2 f} {\partial r_i \partial \theta_i} \right )^{\circ} = \sum_{j \neq i} p_{i,j}^{-A} r \sin \left (\frac{2 \pi (j-i)}{N} \right) (1-A/2) = 0$$

$$\frac{\partial^2 f} {\partial \theta_i \partial \theta_j} = m_i m_j r_{i,j}^{-A-2} r_i r_j \left [-A r_i r_j \sin^2(\theta_j - \theta_i) + r_{i,j}^2 \cos(\theta_j - \theta_i)\right ]$$

$$T_{i,j} := \left (\frac{\partial^2 f} {\partial \theta_j \partial \theta_i} \right )^{\circ} =r^2 p_{i,j}^{-A-2} \left [-A r^2 \sin^2\left ( \frac{2 \pi (j-i)}{N} \right ) + p_{i,j}^2 \cos\left ( \frac{2 \pi (j-i)}{N} \right ) \right ] $$

$$= r^{-A+2} u_{i,j}^{-A} \left [-A  \left (1 - \frac{u_{i,j}^2}{4} \right ) +  \left (1 - \frac{u_{i,j}^2}{2} \right ) \right ]$$

$$\frac{\partial^2 f} {\partial \theta_i^2} = m_i \sum_{j \neq i} m_j r_i r_j r_{i,j}^{-A-2} \left [ A r_i r_j \sin^2(\theta_j - \theta_i) - r_{i,j}^2 \cos(\theta_j - \theta_i) \right ]$$

$$T_{i,i} := \left (\frac{\partial^2 f} {\partial \theta_i^2} \right )^{\circ} = \sum_{j \neq i} r^2 p_{i,j}^{-A-2} \left [ A r^2 \sin^2( \frac{2 \pi (j-i)}{N}) - p_{i,j}^2 \cos( \frac{2 \pi (j-i)}{N}) \right ] = - \sum_{j \neq i} T_{i,j}$$

In terms of these newly defined quantities, with respect to the variables $(r_0, r_1, \ldots r_{N-1}, $ $\theta_0, \theta_1, \ldots \theta_{N-1})$, the Hessian is

$$D^2f^{\circ} = \left ( \begin{array}{c|c}
R & W \\ \hline
-W & T \\
\end{array} \right )$$

In a similar way to that in \cite{hampton_splendid}, we can exploit the circulant structure of the Hessian submatrices to compute its eigenvalues.  The submatrices $R$ and $T$ are circulant and symmetric, while $W$ is circulant and anti-symmetric.  Let $C$ be the $N$ by $N$ matrix with $C_{i,j} = e^{2 \mathbb{I} i j /n}$, where $\mathbb{I}= \sqrt{-1}$, and $i,j \in \{0, \ldots, N-1\}$.  This matrix $C$ orthogonally diagonalizes any circulant matrix of the same dimension.  Thus we have

$$H = \left ( \begin{array}{c|c}
C^{-1} & 0 \\ \hline
0 & C^{-1}  \\
\end{array} \right ) \left ( \begin{array}{c|c}
R & W \\ \hline
-W & T \\
\end{array} \right ) \left ( \begin{array}{c|c}
C & 0 \\ \hline
0 & C \\
\end{array} \right ) = \left ( \begin{array}{c|c}
P & S \\ \hline
-S & Q \\
\end{array} \right )$$
in which the subblocks of $H$ are all diagonal, $P$ and $Q$ are real, and $S$ is purely imaginary.  We can express the entries of $P$,$Q$, and $S$ in terms of the first rows of $R$, $T$, and $W$ respectively:

$$P_{i,i} = \sum_{j=0}^{N-1} R_{0,j} C_{j, i} $$

$$ =R_{0,0} +  \sum_{j = 1}^{\left \lfloor{(N-1)/2}\right \rfloor} 2 \cos(\frac{2 \pi i j}{N})R_{0,j} + \left \{ \begin{array}{c}  0 \ \text{for $N$ \ odd } \\ (-1)^i (2r)^{-A} (A-1) \ \text{for $N$ \ even } \end{array} \right .$$
or more explicitly
\begin{align} \label{Pii} P_{i,i} = & r^{-A} \sum_{j = 1}^{\left \lfloor{(N-1)/2}\right \rfloor} u_{0,j}^{-A} \left ( u_{0,j}^2 \left [ \frac{A}{2} + 1 + (\frac{A}{2} - 1)\cos(\frac{2 \pi i j}{N}) \right]  -2 + 2 \cos(\frac{2 \pi i j}{N}) \right )  \\ 
& + \left \{ \begin{array}{c}  0 \ \text{for $N$ \ odd } \\ (2r)^{-A}((-1)^i(A-1) + (A+1)) \ \text{for $N$ \ even } \end{array} \right . \nonumber
\end{align}

$$Q_{i,i} = \sum_{j=0}^{N-1} T_{0,j} C_{j,i}  = T_{0,0} + \sum_{j = 1}^{\left \lfloor{(N-1)/2}\right \rfloor} 2 \cos(\frac{2 \pi i j}{N}) T_{0,j} + \left \{ \begin{array}{c}  0 \ \text{for $N$ \ odd } \\ (-1)^{i}r^2(2r)^{-A}\ \text{for $N$ \ even } \end{array} \right .$$

which can be written as

\begin{equation} \label{Qii} Q_{i,i} = 2 r^{-A+2} \sum_{j=1}^{\left \lfloor{(N-1)/2}\right \rfloor} u_{0,j}^{-A} \left ( 1 - \cos \left ( \frac{2 \pi i j}{N} \right ) \right ) \left ( \frac{A-2}{2} \left (1 + \cos \left (\frac{2 \pi j}{N} \right ) \right ) + 1 \right ) .
\end{equation}
The latter form of $Q_{i,i}$ makes it clear that $Q_{i,i} > 0$ for $A\ge 2$ (each term of the sum is non-negative).  

Finally

$$S_{i,i} = \sum_{j=0}^{N-1} W_{0,j} C_{j,i} = \mathbb{I} \sum_{j = 1}^{\left \lfloor{(N-1)/2}\right \rfloor} 2 \sin(\frac{2 \pi i j}{N}) W_{0,j}$$
so
\begin{equation} \label{Sii}
S_{i,i}  =  2 r^{-A+1} \frac{A-2}{2} \ \mathbb{I} \ \sum_{j = 1}^{\left \lfloor{(N-1)/2}\right \rfloor} \sin(\frac{2 \pi i j}{N}) \sin(\frac{2 \pi j}{N}) u_{0,j}^{-A}
\end{equation}
where $\mathbb{I} = \sqrt{-1}$.

Now we can compute the eigenvalues of the Hessian of $f$ in pairs from the two by two matrices 
\begin{equation} \label{Eidef}
E_i =  \left ( \begin{array}{c|c}
P_{i,i} & S_{i,i} \\ \hline
-S_{i,i} & Q_{i,i} \\
\end{array} \right )
\end{equation}
We will denote the two eigenvalues of this block by
$$\lambda_{(N, i, \pm)} = \frac{1}{2} \left ( P_{i,i} + Q_{i,i} \pm \sqrt{(P_{i,i} - Q_{i,i})^2 - 4S_{i,i}^2} \right )$$

\subsection{The regular polygon in the vortex case}

Now we consider the regular polygon configurations, starting with the extreme case of $A=2$.

\begin{theorem}
For $A=2$, $Q_{i,i}|_{A=2} = \left \lfloor \frac{i}{2}N - \frac{i^2}{2} \right \rfloor $.
\end{theorem}

\begin{proof}
We will prove this by induction on $i$.  First note that we can specialize equation (\ref{Qii}) for $A=2$ to

\begin{equation} \label{QiiA2}
Q_{i,i}|_{A=2} = \sum_{j=1}^{\lfloor \frac{N-1}{2} \rfloor} \frac{1 - \cos(\frac{2 \pi i j}{N})}{1 - \cos(\frac{2 \pi  j}{N})}
\end{equation}

The base cases we need are for $i \in \{0,1,2,3\}$.  The first two: 
$$Q_{0,0}|_{A=2} = 0$$
$$Q_{1,1}|_{A=2} = \lfloor \frac{N-1}{2} \rfloor$$
follow directly from (\ref{QiiA2}).  For $i=2$ and $i=3$ we rewrite the numerator in terms of $\cos(\frac{2 \pi  j}{N})$, and in each case the denominator appears as a factor we can cancel.  

$$Q_{2,2}|_{A=2} = \sum_{j=1}^{\lfloor \frac{N-1}{2} \rfloor} \frac{1 - \cos(\frac{4 \pi j}{N})}{1 - \cos(\frac{2 \pi  j}{N})} = \sum_{j=1}^{\lfloor \frac{N-1}{2} \rfloor} \frac{2 - 2\cos(\frac{2 \pi j}{N})^2}{1 - \cos(\frac{2 \pi  j}{N})}  $$
$$= 2 \sum_{j=1}^{\lfloor \frac{N-1}{2} \rfloor} \left ( 1 + \cos(\frac{2 \pi j}{N}) \right )  =  N - 2 $$

$$Q_{3,3}|_{A=2} = \sum_{j=1}^{\lfloor \frac{N-1}{2} \rfloor} \frac{1 - \cos(\frac{6 \pi j}{N})}{1 - \cos(\frac{2 \pi  j}{N})} = \sum_{j=1}^{\lfloor \frac{N-1}{2} \rfloor}  (1 + 2\cos(\frac{2 \pi  j}{N}))^2  =  \left \lfloor \frac{3N - 9}{2} \right \rfloor  $$

The final form in each of the above cases can be summed using standard properties of Chebyshev polynomials.  We will briefly review some of the properties of Chebyshev polynomials used there and in what follows.  We define $T_j( \cos(\theta)) = \cos( j \theta)$, and $U_j (\cos(\theta)) \sin(\theta) = \sin((j+1) \theta)$.  The polynomials $T_j$ and $U_j$ satisfy many known relations including the composition formula $T_j(T_k(\theta)) = T_{jk} (\theta)$, and the summation formulae

$$\sum_{j=0}^m T_{2j+1}(x) = \frac{U_{2m+1} (x)}{2}$$
$$\sum_{j=0}^m T_{2j}(x) = \frac{U_{2m} (x) + 1}{2}$$

For the induction step we consider the double difference $$D_i = (Q_{i+2,i+2} - Q_{i,i}) - (Q_{i,i} - Q_{i-2,i-2}) = Q_{i+2,i+2} - 2 Q_{i,i} + Q_{i-2,i-2} $$
After writing all the cosines in terms of $\cos(\frac{2 \pi  j}{N})$, we can eventually simplify $D_i$ to

$$D_i = 4 \sum_{j=1}^{\lfloor \frac{N-1}{2} \rfloor}  (1 + \cos(\frac{2 \pi  j}{N}))\  T_i(\cos(\frac{2 \pi  j}{N})) = -4$$

Now we can conclude the induction; assume that  $Q_{i,i}|_{A=2} =\left \lfloor \frac{i}{2}N - \frac{i^2}{2} \right \rfloor $ for $i < j + 2$.  Then

$$Q_{j+2,j+2}|_{A=2} = 2 Q_{j,j}|_{A=2} - Q_{j-2,j-2}|_{A=2} $$
$$= 2  \left \lfloor \frac{j}{2}N - \frac{j^2}{2} \right \rfloor - \left \lfloor \frac{j-2}{2}N - \frac{(j-2)^2}{2} \right \rfloor - 4 = \left \lfloor \frac{j+2}{2}N - \frac{(j+2)^2}{2} \right \rfloor $$

\end{proof}

\begin{theorem} 
For $A=2$, 
$$P_{i,i}|_{A=2} = (2-i)N + \frac{(i^2-i)N}{N-1}$$
\end{theorem}
\begin{proof}

From (\ref{Pii}) we have
$$P_{i,i}|_{A=2} = r^{-2} \sum_{j=1}^{\lfloor \frac{N-1}{2} \rfloor} u_{0,j}^{-2} \left ( 2 u_{0,j}^2  - 2 + 2 \cos(\frac{2 \pi  i j}{N}) \right ) + \left \{ \begin{array}{c}  0 \ \text{for $N$ \ odd } \\ (2r)^{-2}((-1)^i + 3) \ \text{for $N$ \ even } \end{array} \right .$$

$$= \left (\frac{2N}{N-1} \right ) \left [ \sum_{j=1}^{\lfloor \frac{N-1}{2} \rfloor} \left ( 2 - \frac{1 -  \cos(\frac{2 \pi  i j}{N})}{1 -  \cos(\frac{2 \pi  j}{N})} \right ) + \left \{ \begin{array}{c}  0 \ \text{for $N$ \ odd } \\ ((-1)^i + 3)/4 \ \text{for $N$ \ even } \end{array} \right . \right ]$$

$$ = \left (\frac{2N}{N-1} \right ) \left [ 2 \lfloor \frac{N-1}{2} \rfloor - Q_{i,i} + \left \{ \begin{array}{c}  0 \ \text{for $N$ \ odd } \\ ((-1)^i + 3)/4 \ \text{for $N$ \ even } \end{array} \right . \right ]$$
and the result follows easily from the previous theorem once we consider all the particular cases of $N$ and $i$ being odd or even.
\end{proof}

So it becomes possible to completely determine the eigenvalues of the Hessian for the regular polygon configuration in the vortex case.

\begin{theorem}\label{VortTheorem}
The eigenvalues of the Hessian of $f$ in the case $A=2$ are $(2-i)N + \frac{(i^2-i)N}{N-1}$ and $\left \lfloor \frac{i}{2}N - \frac{i^2}{2} \right \rfloor$ for $0 \le i \le N-1$.  In the quotient configuration space $\mathcal{C}_N$, the regular polygon has Morse index of $0$ for $N \in \{3,4,5,6\}$, it is degenerate for $N=7$, and has a Morse index of $N-5$ for $N \ge 8$.
\end{theorem}
\begin{proof}

It is immediate from the general formula (\ref{Sii}) for $S_{i,i}$ that $S_{i,i}|_{A=2} = 0$, so the eigenvalues of the Hessian are simply $Q_{i,i}|_{A=2}$ and $P_{i,i}|_{A=2}$ as given in the previous theorems.  The remainder of the theorem follows from considering the sign of these expressions: since the eigenspace blocks are orthogonal the Morse index is simply the number of negative eigenvalues of all the blocks.

\end{proof}
The fact that the heptagon has a degenerate Hessian in the vortex case is interesting in comparison to the dynamical stability results in \cite{CabralSchmidt}, in which the heptagon was also a degenerate case.

\subsection{The regular polygon for $A > 2$}

Now we prove a theorem characterizing the Morse index of the regular polygon for larger values of $A$.  A previous related result specialized to the Newtonian ($A=3$) case is that the regular $N$-gon does not have index $0$ for $N \ge 6$ \cite{meyer1988bifurcations, slaminka1990central, woerner1990n}.

\begin{lemma}
For $N \ge 2$, 
$$\lambda_{(N,0,+)} > NA, \ \ \ \ \lambda_{(N,0,-)} = 0$$
$$\text{and } \ \  \lambda_{(N,1,\pm)} > 0.$$
\end{lemma}
\begin{proof}
The result for $i=0$ follows immediately from our expressions for $P_{0,0}$, $Q_{0,0}$, and $S_{0,0}$.

For $i=1$, we can calculate the determinant

\begin{align*}
P_{1,1} Q_{1,1}+S_{1,1}^2 = &  r^{-2A+2} \sum_{j=1}^{\lfloor \frac{N-1}{2} \rfloor} \sum_{k=1}^{\lfloor \frac{N-1}{2} \rfloor} u_{0,j}^{-A} u_{0,k}^{-A} (1- \cos(\frac{2 \pi j}{N}))(1 - \cos(\frac{2 \pi k}{N})) \left \{ \right . \\
&  (A-2)^2 (\cos(\frac{2 \pi j}{N}) - \cos(\frac{2 \pi k}{N})) (\cos(\frac{2 \pi j}{N}) + 1) \\
& \left . + \ 4 (A-2) (\cos(\frac{2 \pi j}{N}) + 1) + 4 \right \}
\end{align*}
and we see the coefficients of $(A-2)^2$ will cancel out (under the interchange of $j$ and $k$), and the remaining terms are positive.  Thus the two eigenvalues are either both negative or both positive.  The trace can be simplified into a form that makes it clearly positive:
$$P_{1,1} + Q_{1,1} = r^{-A} \left ( r^2  + 2 \right ) \sum_{j=1}^{\lfloor \frac{N-1}{2}\rfloor} u_{0,j}^{-A} \left (1 - \cos \left (\frac{2 \pi j}{N} \right ) \right )  \left ( \left (A-2 \right )\left (1 + \cos \left (\frac{2 \pi j}{N} \right ) \right ) + 2 \right ) $$
so the two eigenvalues of the $i=1$ block must be positive.
\end{proof}

For $i \in \{1, \ldots N-1 \}$, $\lambda_{(N,i,\pm)} = \lambda_{(N,N-i,\pm)}$ for each sign choice and so we can restrict our attention to $1 \le i \le \lfloor \frac{N}{2} \rfloor$.  Apart from the special cases of small $N$, it turns out that the interesting eigenvalue of the Hessian is always $\lambda_{(N,2,-)}$ (which equals $\lambda_{(N,N-2,-)}$).  

\begin{theorem} 	
For each $N \ge 5$ there is a value of $A=A_N$, and $i \in \{1, \ldots \lfloor \frac{N}{2} \rfloor \}$, such that the eigenvalues of the Hessian of $f$ for the equal-mass regular polygon satisfy
$$\lambda_{(N,i,+)} > 0,$$
and
$$\lambda_{(N,i,-)}  > 0  \text{ for } i < 2, \ \ \ \lambda_{(N,i,-)}  < 0  \text{ for } i\ge 2$$

for all $A> A_N$
so the Morse index of $f$ for the regular polygon on the quotient configuration space $\mathcal{C}_N$ is  $N-3$ for $A>A_N$. 
\end{theorem}

\begin{proof}
For the purposes of the Morse index we only need to determine the sign of the eigenvalues of the Hessian of $f$, so for each two by two block $E_i$ (cf. Eq. (\ref{Eidef})) we need to know the sign of $P_{i,i} Q_{i,i}+S_{i,i}^2$.  If this is negative then the eigenvalues will have opposite signs.  So we examine the terms of the sums in our expression for $P_{i,i} Q_{i,i}+S_{i,i}^2$. 

In order to make the following rather large expressions more managable we use the notations $\theta = \frac{2 \pi}{N}$, $c_i = \cos(i \theta)$, and $s_i = \sin(i \theta)$; it will also be convenient to use $B = A - 2$ as well as $A$ because of the structure of $S_{i,i}^2$.  With these we have:

\begin{align*}
P_{i,i} Q_{i,i}+S_{i,i}^2 = & - r^{-2A+2} \sum_{j=1}^{\lfloor \frac{N-1}{2} \rfloor} \sum_{k=1}^{\lfloor \frac{N-1}{2} \rfloor} u_{0,j}^{-A} u_{0,k}^{-A} \left \{ 4 (1 - u_{0,j}^2 - c_{ij}) (1 - c_{ik}) \right . \\
+ &  B (1 - c_{ik})(u_{0,j}^2 u_{0,k}^2 - c_{ij} u_{0,j}^2 + c_{ij} u_{0,k}^2 - 5 u_{0,j}^2 - u_{0,k}^2 - 4 c_{ij} + 4) \\
+ &  \frac{B^2}{4} \left ( 4 s_{ij} s_{ik} s_{j} s_{k} -c_{ij} c_{ik} u_{0,j}^2 u_{0,k}^2 + c_{ij} u_{0,j}^2 u_{0,k}^2 - c_{ik} u_{0,j}^2 u_{0,k}^2  \right . \\
+ & \left . \left . 4 c_{ij} c_{ik} u_{0,j}^2 + u_{0,j}^2 u_{0,k}^2 - 4 c_{ij} u_{0,j}^2 + 4 c_{ik} u_{0,j}^2 - 4 u_{0,j}^2  \right ) \right \}
\end{align*}
which fortunately simplifies a little after using the fact that $u_{0,j}^2 = 2 - 2 \cos(j \theta) = 2 - 2 c_j$:

\begin{align*}
P_{i,i} Q_{i,i}+S_{i,i}^2 = &  - r^{-2A+2}  \sum_{j=1}^{\lfloor \frac{N-1}{2} \rfloor} \sum_{k=1}^{\lfloor \frac{N-1}{2} \rfloor} u_{0,j}^{-A} u_{0,k}^{-A} \left \{ -4 (1 - c_{ik}) (1 + c_{ij} -  2 c_{j}) \right . \\
+ &  \ 2 B (1 - c_{ik})(c_{ij} c_{j} - c_{ij} c_{k} + 2 c_{j} c_{k} - 2 c_{ij} + 3 c_{j} - c_{k} - 2) \\
+ &  \left . B^2 \left ( s_{ij} s_{ik} s_{j} s_{k} -(1 + c_{ij}) (1 - c_{ik}) (1 - c_{j}) (1 + c_{k})   \right ) \right \}
\end{align*}

Next we separate out the diagonal terms, since the $B^2$ coefficient vanishes for those, and we want to estimate the leading term:
\begin{align*}
P_{i,i} Q_{i,i}+S_{i,i}^2 = &  - r^{-2A+2}  \left [ \sum_{j=1}^{\lfloor \frac{N-1}{2} \rfloor } u_{0,j}^{-2A} \left \{ -4(1 - c_{ij}) (1 + c_{ij} - 2 c_j) \right . \right . \\
& \left . + 4 B (1 - c_{ij}) (c_j^2 - c_{ij} + c_j - 1) \right \} \\
& + \sum_{j=1}^{\lfloor \frac{N-1}{2} \rfloor} \sum_{k=1, k \neq j}^{\lfloor \frac{N-1}{2} \rfloor} u_{0,j}^{-A} u_{0,k}^{-A} \left \{ -4 (1 - c_{ik}) (1 + c_{ij} -  2 c_{j}) \right . \\
+ &  \ 2 B (1 - c_{ik})(c_{ij} c_{j} - c_{ij} c_{k} + 2 c_{j} c_{k} - 2 c_{ij} + 3 c_{j} - c_{k} - 2) \\
+ & \left . \left . B^2 \left ( s_{ij} s_{ik} s_{j} s_{k} -(1 + c_{ij}) (1 - c_{ik}) (1 - c_{j}) (1 + c_{k})   \right ) \right \} \right ]
\end{align*}

For a fixed $N$, the leading term ($j=k=1$) expanded in $\theta$ dominates for large enough $A$:

$$P_{i,i} Q_{i,i}+S_{i,i}^2  = - r^{-2A+2} u_{0,1}^{-2A} \left [ (i^2 - 3) (A-2) i^2 \theta^{4} + O(N^2 A^2 \theta^{A+4} ) \right ]$$
and for such $A$ it is negative for $i \ge 2$.

\end{proof}

Our numerical investigations suggest a stronger version which we have been unable to prove.

\begin{conj} \label{MainConj}
We conjecture that the previous theorem can be strengthened to say that there exists a unique value $A_N$ for each $N>4$ such that for $A < A_N$ the Morse index of $f$ for the regular polygon on the quotient configuration space $\mathcal{C}_N$ is  $N-5$, and the Morse index is $N-3$ for $A>A_N$. Furthermore, the $A_N$ are monotonically decreasing in $N$ with $\lim_{N \rightarrow \infty} A_N = 2$.
\end{conj}
To illustrate this conjecture a little more precisely we found an ad-hoc Pad{\'e} approximation to $A_N$ which appears to have a relative error of less than 1\% for $5 \le N \le 200$:
$$A_N \approx \frac{2 \, N^{3} - 2.46 \, N^{2} + 0.713 \, N - 91.5}{N^{3} - 3.3\, N^{2} - 17.17 \, N + 58.5}$$

It is an interesting curiousity that  $A_7$ (from Conjecture \ref{MainConj}) is exactly equal to $4$; we did not prove this in detail (i.e. it remains to show the uniqueness of the zero eigenvalue for $A=4$) but the key fact is not difficult to show:
\begin{lemma}
$\lambda_{(7,2,-)}|_{A=4}$ is exactly equal to $0$.
\end{lemma}
\begin{proof}
 This is a straightforward calculation; we have the explicit formulae  \ref{Pii}, \ref{Qii}, and \ref{Sii}, from which we can compute the eigenvalues of the Hessian of $f$.  In the case that $N=7$, these expressions are in terms of trigonometric functions of multiples of $\frac{\pi}{7}$.  These quantities can be calculated as cubic roots (and square roots of cubic roots), for example:
 $$\cos(\frac{\pi}{7}) = {\left(\frac{7}{144} \mathbb{I} \, \sqrt{3} - \frac{7}{432}\right)}^{\frac{1}{3}} + \frac{7}{36 \, {\left(\frac{7}{144} \mathbb{I} \, \sqrt{3} - \frac{7}{432}\right)}^{\frac{1}{3}}} + \frac{1}{6}.$$
These identities let us simplify:

 $$P_{2,2}|_{A=4, N=7} = \frac{7}{4},$$
 $$Q_{2,2}|_{A=4, N=7} =\sqrt{\frac{4375}{8}}$$
 and 
 $$S_{2,2}|_{A=4, N=7} =  \mathbb{I} \sqrt[4]{\frac{214375}{128}}$$

and so $\lambda_{(7,2,-)}|_{A=4} = \frac{1}{2}(P_{2,2}+Q_{2,2} \sqrt{(P_{2,2}-Q_{2,2} )^2 - 4 S_{2,2}^2})|_{A=4, N=7} = 0$.
 
\end{proof}

\section{Equal mass central configurations for small $N$}

In this section we survey equal mass central configurations for small $N$ (up to $N=9$) for our variable homogeneous potential.  Apart from some new results in the four-body problem, the section primarily contains conjectures.

\subsection{The three-body problem}
For the Newtonian three-body problem the central configurations are well known, being characterized by Euler \cite{Euler_col} and Lagrange \cite{Lagrange_3} for all positive masses.  Very little about these configurations changes as the potential exponent is changed (i.e. for $A \in [2,\infty)$): the equilateral triangle is always a central configuration and a minimum for $f$, and there is a symmetric collinear central configuration with index $1$.  There are two distinct equivalence classes of equilateral triangles (multiplicity 2), and three distinct equivalence classes of collinear configurations (multiplicity 3).  Configurations are considered equivalent if there is an orientation-preserving isometry (rigid motion) between them.  For all $A \in [2,\infty)$, the Poincar{\'e} polynomial of the reduced configuration space is $P(t) = 1+2t$, and the Morse polynomial of $f$ is $M(t) + (1+t) = 2 + 3t$.  

\subsection{The four-body problem}

Although unsolved problems remain, for the Newtonian ($A=3$) and vortex case ($A=2$) of the four-body problem the central configurations are well understood, with many particular results for configurations with some special symmetry or other properties \cite{lehmann1891ueber, dziobek_ueber_1900, MacMillanBartky32, Gannaway81, Arens, oneil_stationary_1987, hampton02, long2002four, hampton_co-circular_2005, hampton_finiteness_2005, leandro_central_2006, perez2007convex, albouy_symmetry_2008, hampton_finiteness_2009, pina2010central, shi2010classification, cors2012four, xie2012isosceles, alvarez2013symmetric, barros2014bifurcations, hampton_vort_pairs2014, erdi2016central, deng2017four, fernandes2017convex, corbera2018four, santoprete2018four}.  The equal mass case is especially well characterized.  The investigations of Sim{\'o} \cite{simo_relative_1978} strongly indicate that the equal mass case has the largest number of central configurations for the Newtonian case, although a formal proof of that is still unavailable.  Some of the bifurcations found by Sim{\'o} have been rigorous analyzed more recently \cite{rusu2016bifurcations}.

Albouy \cite{albouy_1995} proved for a rather general potential function (which includes our homogeneous potential for $A \in [2,\infty)$) that for four bodies of equal mass the planar central configurations always have at least an axis of symmetry.  For the special cases of $A=2$ and $A=3$ he completely characterized the central configurations \cite{albouy_1996}.  For $A=2$,  the square is the only strictly planar convex central configuration, and the equilateral triangle with a central fourth mass is the only concave central configuration.  For $A=3$ there is a second concave central configuration with a central mass on the axis of symmetry of an isosceles triangle.  Albouy conjectured that for $A>2$ there are no additional central configurations compared to the $A=3$ case.  In this section we show this conjecture is true at least for $A>3$, and furthermore characterize the Morse indices of $f$ of the equal-mass four-body central configurations for all $A>3$.  These configurations are shown for $A=3$ and $A=20$ in Figure (\ref{fourbodyccs}).

\begin{figure}
\includegraphics[width=4.8in]{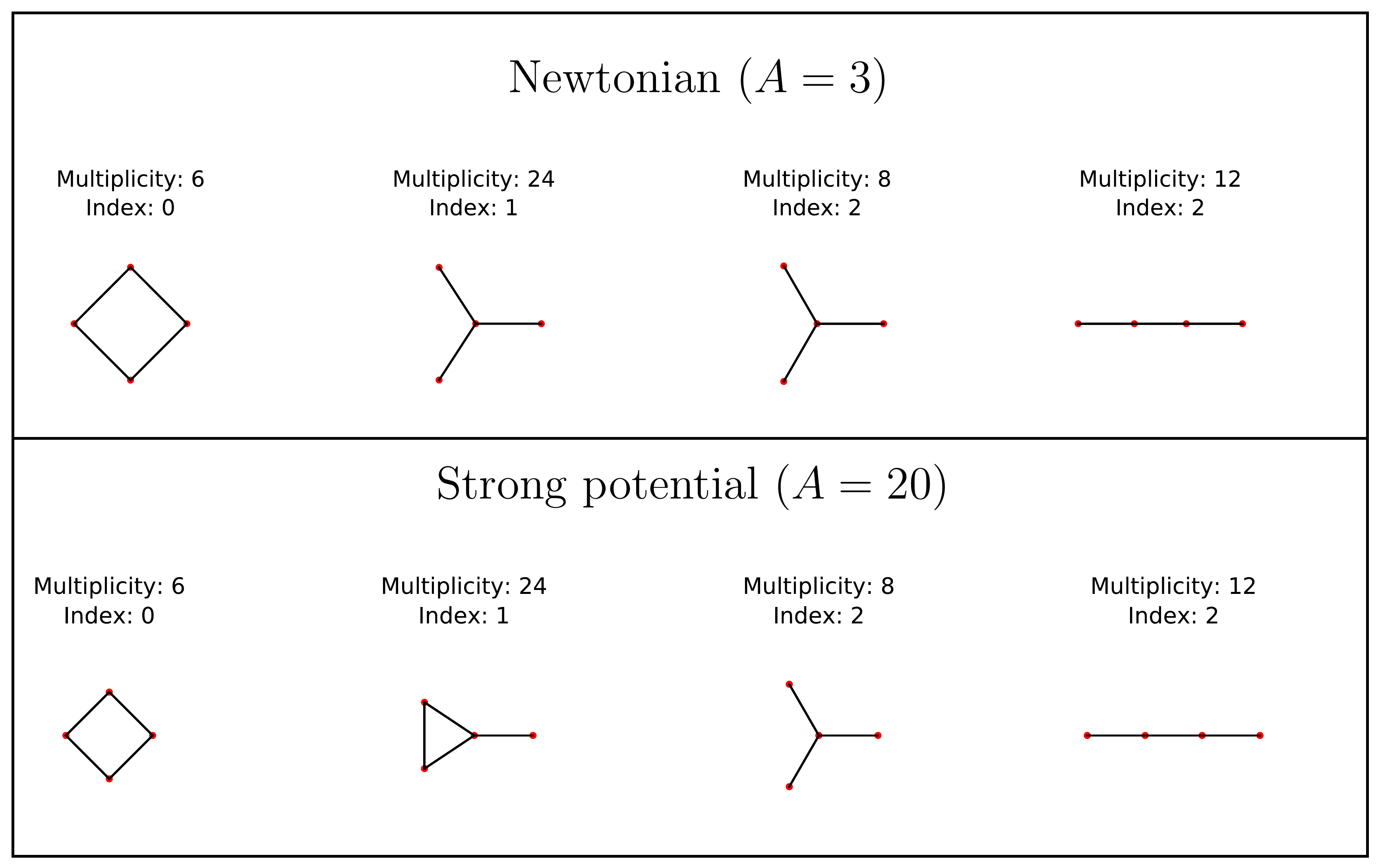}
\caption{Central configurations of the equal-mass four-body problem for $A=3$ and $A=20$.}
\label{fourbodyccs}
\end{figure}

For the regular polygon central configuration with $N=4$ (the square), we can compute the radius $r$:
$$r = \left ( \frac{2(\sqrt{2})^{2-A}+ 2^{2-A}}{8} \right )^{\frac{1}{A}} = \left ( 2^{-1-A} [2^{A/2} + 1 ] \right )^{\frac{1}{A}}$$
(this is a special case of the general result given in Section \ref{regpolysec}).

We can find the Morse index of the square by explicitly computing the eigenvalues of the Hessian of $f$.  The first two of the eight eigenvalues of the Hessian of $f$ are $\lambda_{(4, 0, +)} = 4A$ and $\lambda_{(4, 0, +)} = 0$.  The other six are more complicated.  There are two equal pairs $\lambda_{(4, 1, \pm)} = \lambda_{(4, 3, \pm)}$, for which

$$P_{1,1} = P_{3,3} = 2 \left ( \frac{2^{A/2}A + 2}{1 + 2^{A/2}} \right )$$

$$Q_{1,1} = Q_{3,3} = \frac{{\left(2^{\frac{1}{2}A+ 1} A + 4\right)} \left({\left(2^{\frac{1}{2} \, A} + 2^{A}\right)} 2^{-\frac{3}{2} \, A - 1}\right)^{\frac{2}{A}}}{1+ 2^{A/2}}$$
$$S_{1,1} = S_{3,3} = \frac{2 (A - 2) ((2^{A/2}+ 2^A) 2^{-3A/2 - 1})^{1/A}}{1 + 2^{-A/2}}$$

and with these in hand it is not difficult to show that $\lambda_{(4, 1, \pm)} = \lambda_{(4, 3, \pm)} > 0$ for all $A \ge 2$. 
	
Finally $\lambda_{(4,2,\pm)}$ is determined by 

$$P_{2,2} = \frac{4 A}{1 + 2^{A/2} }$$
$$Q_{2,2}  =\frac{2^{A/2 + 2} \left({\left(1+ 2^{A/2}\right)} 2^{- A - 1}\right)^{\frac{2}{A}} A}{1+ 2^{A/2}}$$
$$S_{2,2}  = 0$$
and since the diagonal entries of this block are always positive $\lambda_{(4,2,\pm)} > 0$ for all $A \ge 2$.

Thus in the quotient space $\mathcal{C}_4$  the eigenvalues are positive and the square is always a minimum of $f$.  The same conclusion was reached by Jersett \cite{Jersett2018} using different methods.  Under the direct isometry equivalence relation there are 6 distinct labelings of the square, so it has multiplicity 6.

The equilateral triangle with a mass at its center was also studied in this context by Jersett \cite{Jersett2018}.  It has eigenvalues $0, 4, 4, 4A$ and two pairs of eigenvalues

$$\lambda_{\pm} = \frac{6 + 3A + (5A-6)3^{A/2}}{3 + 3^{A/2}} \pm \frac{\sqrt{(16A^2 + 9(A-2)^2)3^A + 9(A-2)^2(1- 2\cdot 3^{A/2} )}}{3 + 3^{A/2}} $$
The pair of eigenvalues $\lambda_{-}$ are negative \cite{Jersett2018} for all $A\ge2$, so the Morse index of this configuration is always 2.  Any of the four masses can be at the center, and then there are only 2 distinct ways to label the outer triangle (under the orientation-preserving equivalence relation), so the equilateral triangle with a mass at its center has multiplicity 8.

The isosceles central configurations of the four-body problem, which have two pairs of equal mutual distances, are unfortunately much more complicated to analyze.  Its lack of rotational symmetry means it has multiplicity 24.

Using the Albouy-Chenciner equations \cite{albouy_probleme_1997, LMS} for the isosceles configuration, we excluded most of the mutual distance and $A$-parameter space using interval analysis.  After refining the parameter intervals, we then also used a method from \cite{Shary_2014} to prune interval sets which could not contain a bifurcation (i.e. where the Jacobian of our system must have maximal rank).  To summarize this method (Theorem 5 in \cite{Shary_2014}): for an interval matrix $A$ with entries $[\underline{a},\overline{a}]$, we define the midpoint and radius matrices $mid(A) = (\underline{A}+\overline{A})/2$ and $rad(A) = (\overline{A} - \underline{A})/2$.  Then any matrix with entries contained in the interval entries of $A$ has full rank if $\sigma_{max}(rad(A)) < \sigma_{min}(mid(A))$, where the $\sigma$ denote the singular values of the singular value decomposition of each matrix; as the inequality on singular values is an exact result, it needs to be strengthened slightly to account for the computational precision.

This interval arithmetic method worked very well for $A>5$, and sufficiently well to exclude bifurcations for $A>3$.  However, for $A<3$ it became prohibitively computationally expensive due to the bifurcation at $A=2$ where $f$ is no longer nondegenerate.  

Since we know that for $N=4$ the Poincar{\'e}  polynomial of the configuration space $\mathcal{C}_4$ is $P(t) = 1 + 5t + 6 t^2$ \cite{LMS}, the above analysis implies
\begin{theorem}
 For $A \ge 3$, the Morse polynomial of $f$ is 
$$M = P + (1+t)(5+14t) = 6 + 24 t + 20 t^2.$$
\end{theorem}

All of our numerical analysis strongly supports the following conjecture (a slightly stronger version of a speculation in \cite{albouy_1996}), which we are unable to rigorously prove at this time:

\begin{conj}
 For $A > 2$, the Morse polynomial of $f$ is 
$$M = P + (1+t)(5+14t) = 6 + 24 t + 20 t^2.$$
\end{conj}

 For our purposes, the case of four equal masses is a somewhat special case in that we believe the Morse indices of the critical points do not change in the interval $(2,\infty)$, although there is a degeneracy for $A=2$.   We will see below that for $N>4$ there are bifurcations as $A$ is varied.

\subsection{The five-body problem}

Much less is known about 5-body central configurations in general compared to $N=4$.  In the Newtonian case, the earliest systematic attempt was by Williams \cite{williams_permanent_1938}, who attempted to extend the approach that MacMillan and Bartky \cite{MacMillanBartky32} pioneered for $N=4$ on convex configurations for general (not necessarily equal) masses; the work of Williams was later improved by Chen and Hsiao \cite{chen2018strictly}.  There are limited results on configurations with particular symmetries \cite{hampton_stacked_2005, llibre2008new,hampton_finiteness_2010, gidea2010symmetric, llibre2011new, cornelio2017family}.  Albouy and Kaloshin proved that there are finitely many five-body central configurations in the Newtonian case, apart from some exceptional cases determined by polynomials in the mass parameters for which the result is unknown \cite{albouykaloshin}.

For equal masses the central configurations of the five-body problem in the Newtonian case was completely characterized with a homotopy continuation method in \cite{LeeSantoprete2009}.  We can use our formula for the eigenvalues of the Hessian of $f$ to compute the Morse index of the regular pentagon.  Then using numerical results we speculate on the complete Morse structure of the problem for $A \in [2, \infty)$.  

The central configurations for $A=3$, $A=7$, and $A=20$ are shown in Figure \ref{fivebodyccs}.

\begin{figure}

\includegraphics[width=4.5in]{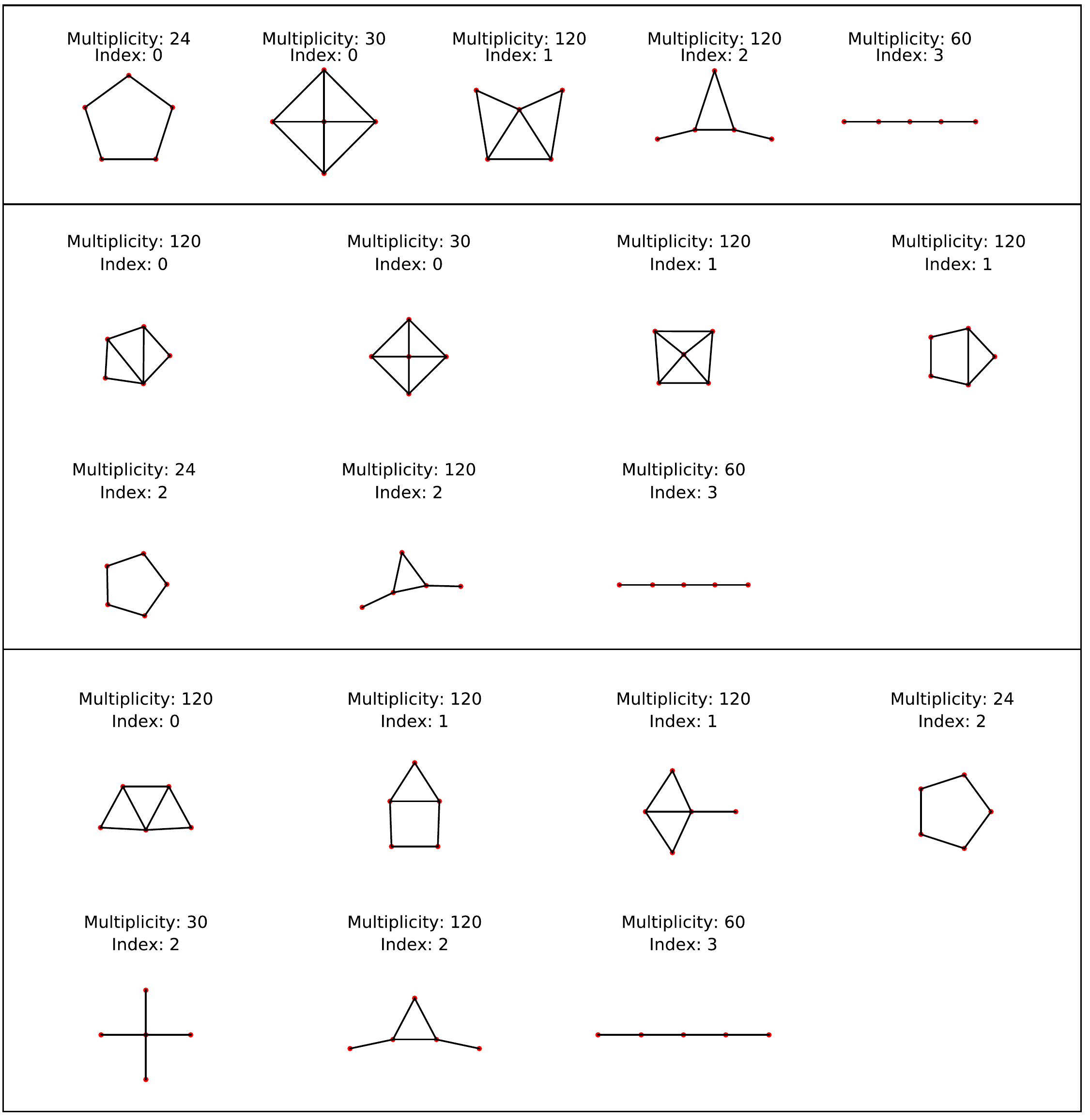}

\caption{Central configurations of the five-body problem for $A=3$, $A=7$, and $A=20$.}
\label{fivebodyccs}
\end{figure}

The Hessian of $f$ for the regular pentagon has a bifurcation for some  $A \in (6.755, 6.756)$.  As $A$ increases through this bifurcation value, the regular pentagon goes from having Morse index 0 to Morse index 2, and two new central configurations are created.  The first new configuration has index 0, and as $A$ increases its shape becomes close to being three equilateral triangles packed in a row (see Figure \ref{fivebodyccs}).  The second new configuration has Morse index 1, and its shape approaches a square topped by an equilateral triangle.  

An interesting bifurcation occurs at  $A \approx 7.5637$.  For $A$ below this bifurcation value, there are index-1 central configurations which are symmetric trapezoids with a fifth mass symmetrically placed in the interior of the trapezoid, and the symmetric cross has index 0.  At the bifurcation the trapezoid becomes a square, and the symmetric cross becomes degenerate.  After the bifurcation (for larger values of $A$) the symmetric cross has index 2, and instead of symmetric trapezoids there are symmetric concave kites with index 1.

We summarize our numerical results by the following conjecture 

\begin{conj}
There are unique values $A_5 \in (6.755, 6.756)$ and $A_c \in ( 7.5636,  7.5638)$ such that for $2 \le A < A_5$, the Morse polynomial of $f$ on $\mathcal{C}_5$  is
$$M(t) = 54 + 120 t + 120 t^2 + 60 t^3 = P(t) + (1+t) (53 + 58 t + 36 t^2)$$
for $A_5 < A < A_c$:
$$M(t) = 150 + 240 t + 144 t^2 + 60 t^3 = P(t) + (1+t) (149 + 82 t + 36 t^2)$$
and finally for $A_c < A$:
$$M(t) = 120 + 240 t + 174 t^2 + 60 t^3 = P(t) + (1+t)(119 + 112 t + 36 t^2)$$
\end{conj}
(For the 5-body problem the Poincar{\'e} polynomial for the reduced configuration space is $P(t) = 1 + 9 t + 26 t^2 + 24 t^3$.)

\subsection{The six-body problem}

In Figure \ref{sixbodyccs} we show central configurations of the six-body problem for $A=3$ and $A=20$.  

\begin{figure}[h!t]
\includegraphics[width=5in]{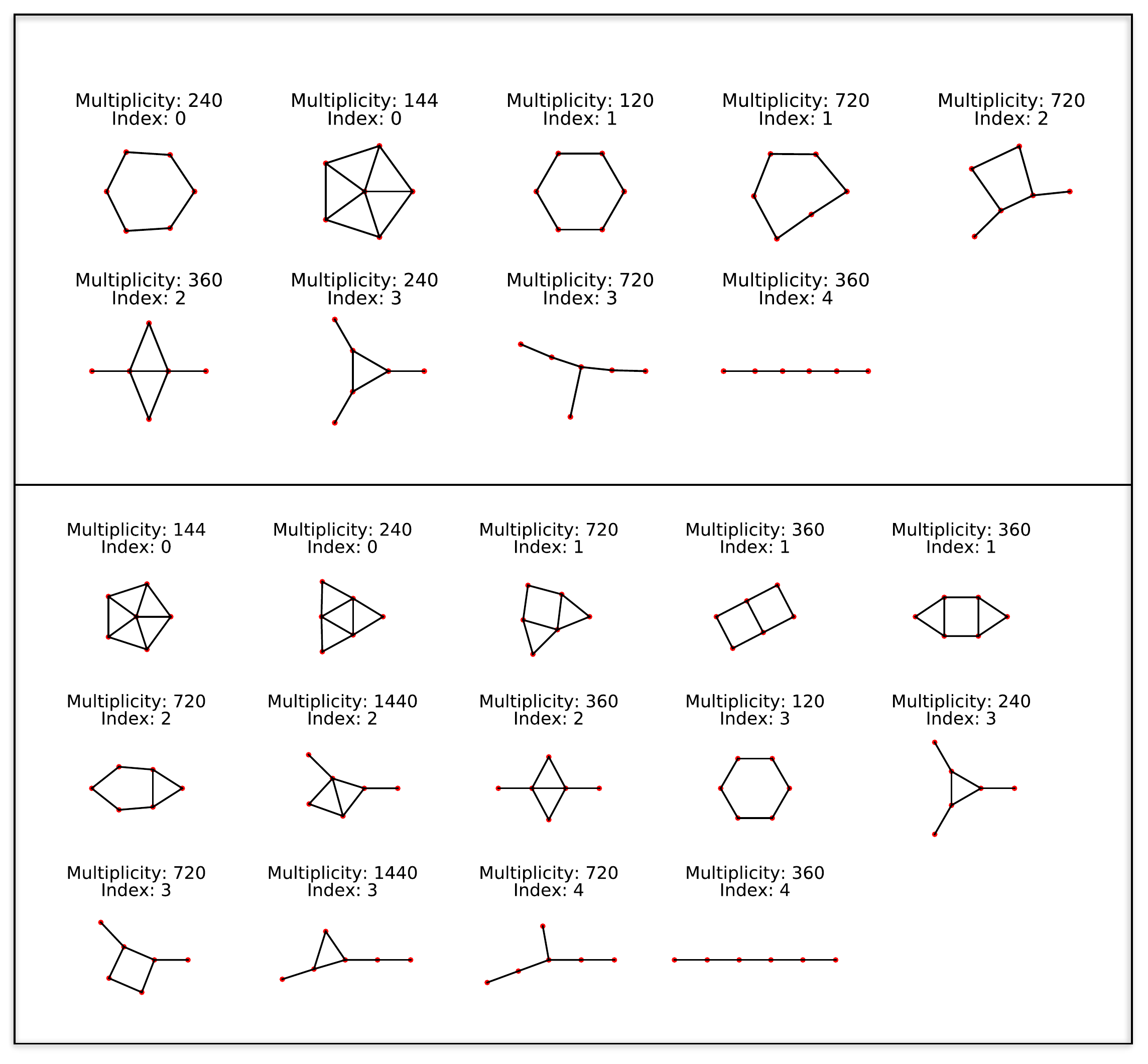}
\caption{Conjectured central configurations of the six-body problem for $A=3$ and $A=20$.}
\label{sixbodyccs}
\end{figure}

The Newtonian configuration close in shape to the regular hexagon is a twisted crown; the existence and uniqueness of the relative equilbria with this type of symmetry has been studied in some detail \cite{moeckel1995bifurcation, yu2012twisted, BarrabasCors}.

In the equal-mass case for $A=2$ and $A=3$, it seems from several numerical experiments that the first time a central configuration without any symmetry appears is $N=8$ \cite{aref1998point,ferrario2002central}; in this context it is interesting that as $A$ increases several asymmetric configurations are created from bifurcations already for $N=6$.  

Let us consider the asymmetric index-2 central configuration, the seventh in Figure (\ref{sixbodyccs}), as a case study in what a theory of central configurations for the limit $A \rightarrow \infty$ might look like.  Our choice of $f$ was motivated partly by the desire that in the limit $A \rightarrow \infty$, the nearest-neighbor distance would approach $1$.  The limiting configuration in question would then be a rhombus composed of equilateral triangles with two masses attached to a single edge.  For large $A$, these masses only effectively interact with their single nearest neighbor.  Assuming that the core rhombus is robust to small perturbations, we need only determine positions for the single-edge masses so that their single interaction direction is parallel to their position (i.e. pointing towards the center of mass).  Denote the rhombus positions by $q_1, \ldots, q_4$, and 
assume $q_5$ only interacts with $q_1$, and $q_6$ with $q_2$, so that $q_5 = q_1 + e^{\mathbb{I}\theta_1}$ and $q_6 = q_2 + e^{\mathbb{I}\theta_2}$ .  Then (assuming equal masses) the additional constraints on this limit configuration are
$$\sum_{i=1}^{6}q_i = 0$$
$$  \mu_1 e^{\mathbb{I}\theta_1} = q_5, \ \ \  \mu_2 e^{\mathbb{I}\theta_2} = q_6$$
where the $\mu_i$ and $\theta_i$ are real.  This can be converted into a polynomial system with $q_i = (x_i, y_i)$, which we solved with computer assistance by computing a Gr\"{o}bner basis using Singular \cite{DGPS} within Sage \cite{sagemath}.  Although these equations are much simpler than those of a central configuration for finite $A$, we were somewhat surprised that they require finding roots of sixth-degree polynomials; for example, the position $y_1$ is a root of 
$$ 11583  + 44505 y_{1} - 9238 y_{1}^{2}   - 71696 y_{1}^{3} + 52212 y_{1}^{4} - 21692 y_{1}^{5} + 7448 y_{1}^{6} =0$$
with $y_1 \approx = 1.33$.

For the six-body problem the Poincar{\'e} polynomial is $P(t) = 1 + 14 t + 71 t^2 + 154 t^3 + 120 t^4$, and corresponding to Figure (\ref{sixbodyccs}) we have the following conjectures:
\begin{conj}
For $6$ bodies, for sufficiently large $A$,
$$M(t) = 384 + 1440 t + 2520 t^2 + 2520 t^3 + 1080 t^4  = P(t) + (1+t)(383+1043 t+1406 t^2+960 t^3)$$
and for $A=3$
$$M(t) = 384 + 840 t + 1080 t^2 + 960 t^3 + 360 t^4   = P(t) + (1+t)(383 + 443 t + 566 t^2 + 240 t^3)$$
\end{conj}

Support for this comes from the independent investigations of Ferrario \cite{ferrario2002central}, who found consistent sets of central configurations in the Newtonian case using a fixed-point method for $N \in \{6,7,8,9\}$.

\subsection{The \{7, 8, 9\}-body problems}

For larger $N$, it becomes difficult to find all of the equal mass central configurations for large $A$.  For the Newtonian case we have the following conjectures which agree with the numerical results of Ferrario \cite{ferrario2002central} (apart from what may be a typo: the 26th central configuration of the 9-body problem listed by Ferrario should have isotropy 1, rather than the $\frac{1}{2}$ stated there, corresponding to a multiplicity of $9!$ rather than $2\cdot 9!$).  In the vortex case ($A=2$) there appear to be exactly 12 central configurations \cite{faugere2012solving}, so at least one bifurcations occur between $A=2$ and $A=3$.

\begin{conj}
The Morse polynomials of $f$ for $A=3$ are
$$\text{for } N=7, \ \ M(t) = 120(7 + 84 t + 132 t^{2} + 105 t^{3} + 84 t^{4} + 35 t^{5})$$
$$\text{for } N=8, \ \ M(t) = 720(8 + 56 t + 224 t^{2} + 301 t^{3} + 210 t^{4} + 112 t^{5} + 28 t^{6})$$
$$\text{for } N=9, \ \ M(t) = 5040(81 + 216 t + 384 t^{2} + 732 t^{3} + 746 t^{4} + 396 t^{5} + 168 t^{6} + 36 t^{7})$$
\end{conj}

These conjectured central configurations are pictured in Figures (\ref{sevenbodyccs}), (\ref{eightbodyccs}), and (\ref{ninebodyccs}).

\begin{figure}
\includegraphics[width=5in]{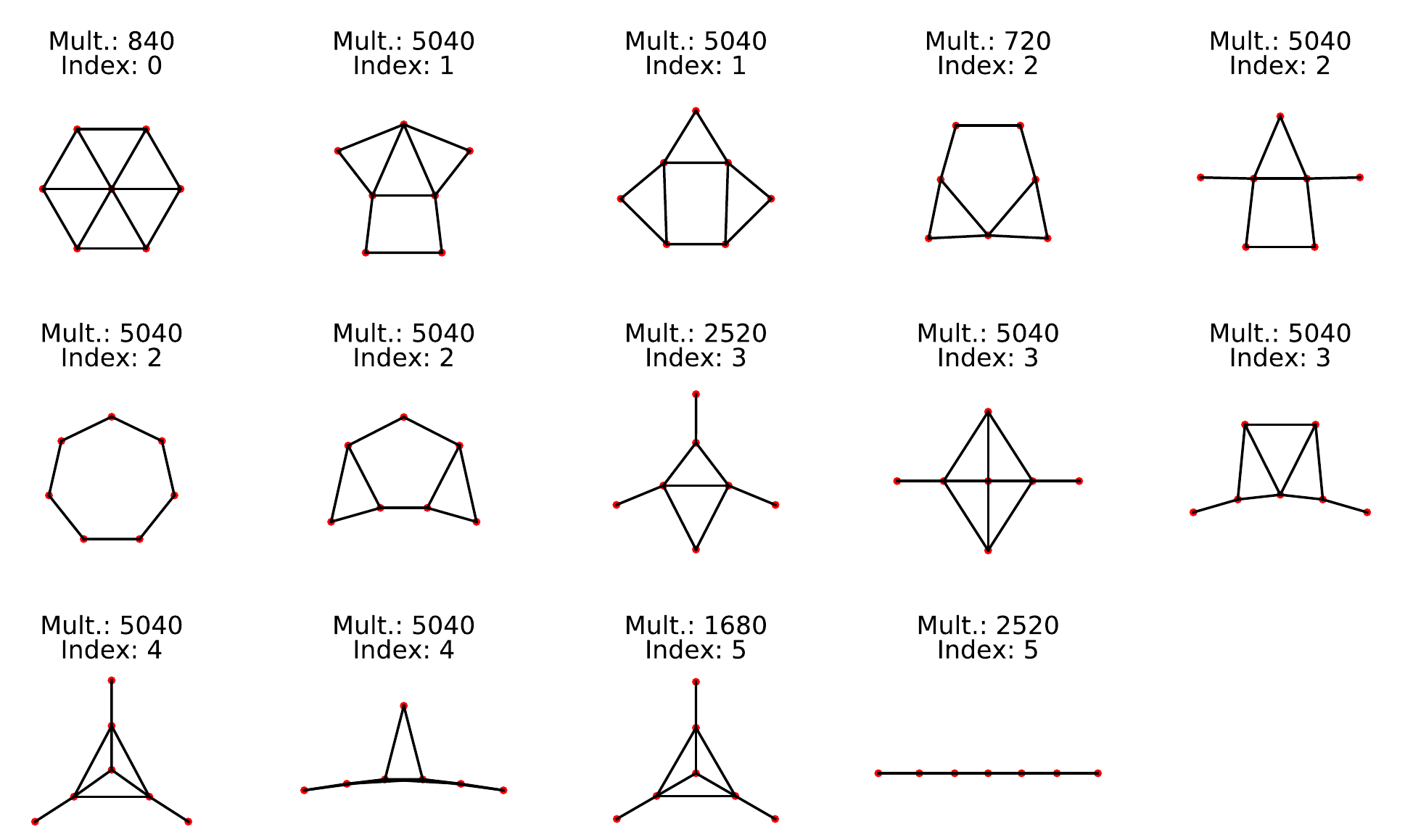}
\caption{Conjectured central configurations of the equal-mass seven-body problem for $A=3$, with Morse indices and multiplicities.}
\label{sevenbodyccs}
\end{figure}

\begin{figure}
\includegraphics[width=5in]{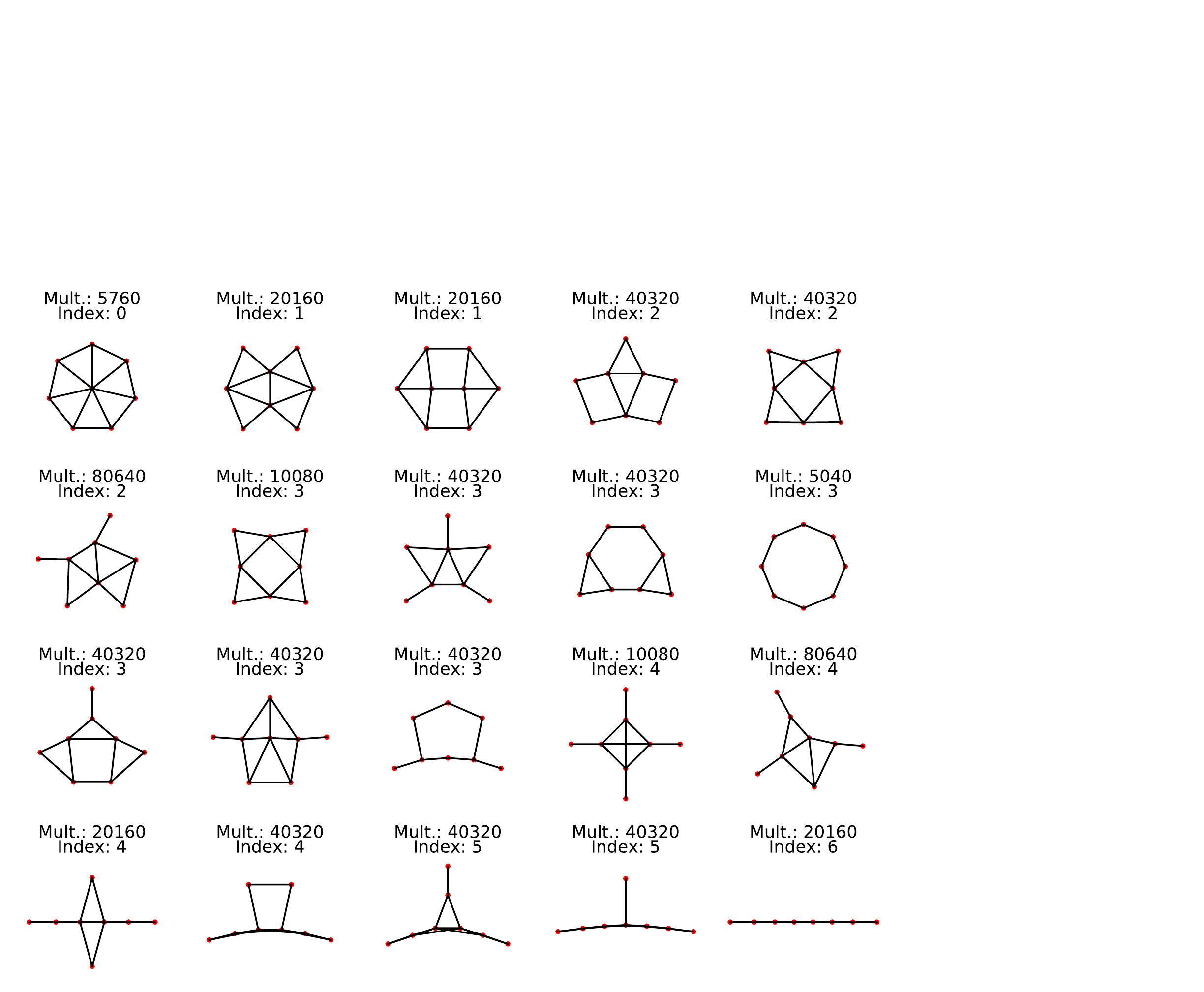}
\caption{Conjectured central configurations of the equal-mass eight-body problem for $A=3$, with Morse indices and multiplicities.}
\label{eightbodyccs}
\end{figure}

\begin{figure}
\includegraphics[width=5in]{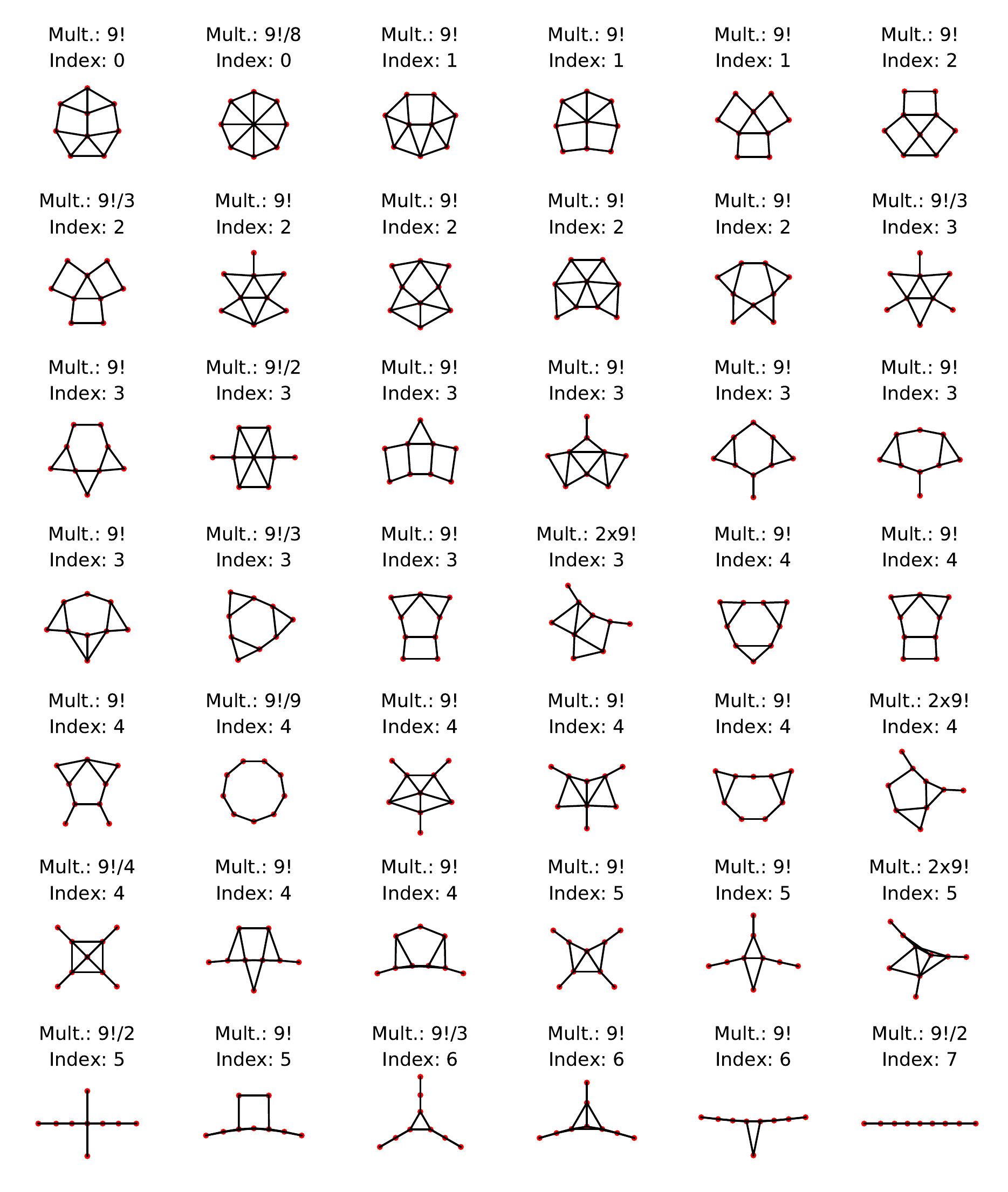}
\caption{Conjectured central configurations of the equal-mass nine-body problem for $A=3$, with Morse indices and multiplicities.}
\label{ninebodyccs}
\end{figure}

Our experience so far has also suggested another conjecture:

\begin{conj}
The number of equal-mass central configurations never decreases as the exponent $A$ increases.
\end{conj}
This conjecture may also be true for unequal mass central configurations, but we lack the intuition to be confident in stating this stronger form.

\section{Collinear central configurations}

Fortunately, results from the Newtonian case on the collinear central configurations can be easily extended to $A \in [2,\infty)$.  The following result is something of a folk theorem, I do not know of a reference that explicitly states it:

\begin{theorem} 
For any $A \in [2,\infty)$, for each ordering of $N$ positive masses on a line there is a unique central configuration, and its (planar) Morse index is $N-2$.
\end{theorem}

\begin{proof}
The uniqueness of the collinear configurations for a given ordering can be proved as an easy generalization of the argument in \cite{LMS} (section 2.9 of that work), which shows that the function $f$ is convex on each connected component of the collinear configuration equivalence classes (an elegant proof improving on the original result of Moulton \cite{moulton_straight_1910}).  There is also a proof by Ferrario for homogeneous potentials \cite{ferrario2002central} using a fixed-point method.  The statement of the Morse index being $N-2$ can be proved by generalizing the creative argument of C. Conley presented in \cite{pacella} and \cite{LMS}, as the exponent being $A=3$ plays no essential role in that proof.  

\end{proof}

The idea of the proof of Conley, which uses an auxiliary dynamical system that converges to collinear configurations, may have inspired a paper of Buck \cite{buck1991collinear} on Newtonian collinear configurations, which would be interesting to generalize to $A \in [2, \infty)$.

\section{Future Directions}

In addition to the various conjectures given in this article we would like to highlight some more general goals.  

\begin{enumerate}

\item A similar analysis to the one given here for the regular polygon could be carried out for the $N+1$ problem of a regular polygon with a central mass.  If the problem is restricted to all equal masses (i.e. including the central mass), the central mass should become inconsequential for large $N$ and $A$. Much is also already known about nested and `twisted' regular polygon configurations \cite{elmabsout1991nouvelles, moeckel1995bifurcation, zhang2002nested, sekiguchi2004bifurcation, lei2006rosette, llibre2009triple, corbera2009existence, celli2011polygonal, yu2012twisted, BarrabasCors, wang2015note, zhao2015central} which would be another relatively easy extension.

\item Numerically complete an analysis of all bifurcations in the equal mass $N$-body problem as the potential exponent $A$ is varied in $[2,\infty)$ for $N \in \{6, \ldots, 10\}$ (and higher if possible).  

\item Find a consistent (within Morse theory) set of central configurations for the equal mass $7$-body problem.

\item Extend any of these results to non-equal masses; even a perturbative analysis near the equal mass case would be a significant advance.  It may also be relatively easy to extend to restricted problems (where some of the masses are infinitesimal compared to others), which already have a rich literature of results in the Newtonian case \cite{Lindlow22,schaub1929speziellen, holtom_permanent_1943, pedersen1944librationspunkte, pedersen1952stabilitatsuntersuchungen, Gannaway81, Arens, Xia91, MoeckelSmall97, cors2004central, leandro_central_2006, santos2007symmetry, scheeres1991linear, barros2014bifurcations, hampton_jensen_R}.

\item Derive equations, or a combinatorial/linear-algebraic framework, for central configurations in the limiting case of $A \rightarrow \infty$.  Compared to the Newtonian case (cf. \cite{buck1990clustering}) it should be much easier to characterize possible central configurations for all $N$.  We strongly believe that the development of such a framework is acheivable and will shed useful light on the problem for all $A \ge 2$.

\end{enumerate}

\section*{Conflicts of interest}

The author declares that he has no conflict of interest.

\bibliography{../CelMechEtc}{}
\bibliographystyle{amsplain}

\end{document}